\documentclass[preprint,12pt]{elsarticle}
\usepackage[margin=2.8cm]{geometry}
\usepackage{amsmath,amsthm}
\usepackage[usenames]{color}
\definecolor{dar}{rgb}{0.6,0.2,0.1}

\newcommand{\tr}{^{\prime}}
\def\ovl#1{\mbox{$\overline #1$}}
\def\b#1{\mbox{\boldmath $#1$}}    
\def\bl#1{\mbox{\footnotesize \boldmath {$#1$}}} 
\def\bt#1{\mbox{\boldmath $#1$}\tr}    
\def\cg#1{\ensuremath{\mathcal{#1}}}      
\def\cgl#1{\mbox{\scriptsize {${\cal #1}$}}}
\def\bs#1#2{\mbox{{#1}$\backslash${#2}}}
\def\bsl#1#2{\footnotesize\mbox{{#1}{$\backslash$}{#2}}}
\newcommand{\bksl}{\scriptsize\mbox{{$\backslash$}}}
\newcommand{\bks}{\mbox{$\backslash$}}

\newcommand{\diag}{{\rm diag}}    

\newtheorem{theorem}{Theorem}
\newtheorem{lemma}{Lemma}
\newtheorem{definition}{Definition}

\newtheorem{example}{Example}
\newtheorem{remark}{Remark}

\newcommand{\ci}{\mbox{\protect $\: \perp \hspace{-2.3ex}\perp$ }}

\begin{document}
\begin{frontmatter}
\title{A class of smooth models satisfying marginal and context specific conditional independencies}
%
\author[1]{R. Colombi}
\address[1]{ Statistical
Laboratory, University of Bergamo, Italy}
\author[2]{A. Forcina}
\address[2]{Dipartimento di Economia, Finanza e Statistica, University of Perugia, Italy}

\markboth{Colombi, Forcina}{Marginal log-linear parameters for
context specific conditional independencies}
\begin{abstract}
We study a class of conditional independence models for  discrete data with the property that one or more log-linear interactions are defined within two different marginal distributions and then constrained to 0; all the conditional independence models which are known to be non smooth belong to this class. We introduce a  new marginal log-linear parameterization and show that smoothness may be restored by restricting one or more independence statements to hold conditionally to a restricted subset of the configurations of the conditioning variables. Our results are based on a specific reconstruction algorithm from log-linear parameters to probabilities and fixed point theory. Several examples are examined and a general rule for determining the implied conditional independence restrictions is outlined.
\end{abstract}
\end{frontmatter}
\section{Introduction}

Conditional independence models for discrete data are determined by a set of constraints on log-linear interactions defined within different marginal distributions of a contingency table. The family of hierarchical and complete marginal log-linear parameterizations (HCMP for short) introduced by \citet{BerRud:02} provides a general framework for combining log-linear constraints defined on a collection of marginal distributions into an overall joint distribution. Methods for determining whether and how a conditional independence model may be translated into a HCMP have been studied by \citet{Rudas:10} and \citet{FoLuMa:10} among others; the fact that a HCMP exists, is a sufficient condition for the model to be smooth.

On the other hand, it is known that no HCMP exists when a model imposes constraints on the same log-linear interaction defined in two different marginals. It has been shown \cite[Theorem 3]{BerRud:02} that, when the same interaction is defined in two different marginals, the jacobian of the mapping from log-linear parameters to probabilities is singular for the uniform distribution. Though, formally, this does not imply that the model itself has singularities, all known models with singularities correspond to cases where no HCMP exists because one or more interactions are constrained more than once. In this paper we study the class of conditional independence models where the same interaction is constrained in two or more marginal distributions and we show, essentially, that any such model is non smooth but can be turned into a smooth model by restricting it to a suitable context specific conditional independence model.

Following \citet{BerRud:02}, we may assume, without loss of generality, that the marginal distributions of interest have been arranged in a non decreasing order and that they will be reconstructed one at a time starting from the smallest. Because the full joint distribution is simply the last marginal in this list, we need only to consider how to determine a given marginal distribution when one or more log-linear interactions to be constrained have already been defined and/or constrained in a previous marginal. A useful tool for reconstructing marginal distributions in a sequence is the mixed parameterization \citep[e.g.,][]{BarNi:79} by which we may  combine the marginal probabilities from previous marginals with the log-linear interactions defined in the marginal distribution under consideration. Because  the mapping produced by the mixed parameterization is one to one and smooth, the question of whether a model is smooth up to a given marginal, is equivalent to the question whether an algorithm based on the mixed parameterization exists and converges. By using results from the theory of fixed point algorithms, we study the jacobian of a new reconstruction algorithm that allows certain log-linear interactions to be redefined and show that this may either converge, and thus the model is smooth, remain at the starting point irrespective of the starting value, implying that the resulting distribution is not uniquely determined by the log-linear parameters or, simply not converge. A formal proof of these properties is derived under complete independence and we provide substantial evidence to support the conjecture that our results hold in the general case.

The results derived in this paper help clarifying which interaction parameters may be redefined and which other interactions should be omitted as a replacement. In particular we show that smoothness is restored only when a specific subset of other interactions is omitted;  these interactions have the property that, when they are missing, and thus unconstrained, the conditional independence of interest holds only on a subset of the configuration of the conditioning variables. Log-linear models which allow context specific conditional independences have been studied in detail by \citet{Hojs:2004} who also derives a markov property for undirected graphs involving context specific conditional independencies. A  special case of the results derived here was considered by \citet{RoLuLa}.

In section 2, we introduce the basic notations, define marginal log-linear interactions and review the properties of the mixed parameterization. In section 3, after presenting a set of motivating examples, we introduce a new algorithm for reconstructing a marginal distribution when interactions defined in previous marginals have to be constrained again and we analyze its convergence properties. In section 4 we study the consequences on the original conditional independence statements of omitting constraints on a specific subset of higher order interactions and show that this results in context specific restrictions.
\section{Notations and preliminary results}
We study the joint distribution of $d$ discrete random variables
where $X_j$, $j=1,\dots , d$, takes values in $(0, \dots,r_j)$. For
conciseness, we denote variables by their indices and use capitals
to denote non-empty subsets of $V$ = $\{1, \dots, d\}$; such subsets
will determine the variables involved either in a marginal
distribution or in an interaction term. The collection of all
non-empty subsets of a set $M \subseteq V$ will be denoted by $\cg
P(M)$.  In the following we write $i_1i_2 \dots i_k$ as a shorthand notation for $\{i_1,i_2 \dots ,i_k\}$.
For a given $M\subseteq V$, the marginal distribution in $M$ is
determined by the cell probabilities $p_M(\b x_M)$ = $P(X_j=x_j,
\forall j\in M)$. We  introduce a shorthand notation that allows
to specify the values of selected subsets of the arguments in a marginal probability and on the log-linear interactions to be defined below. Let $J \subset I\subset M$, then $p_M(\b x_J,\b
x_{\bsl I J},\b x_{\bsl M I})$ denotes the marginal probability where
$\b x_J$ is the value of $X_h,\:h\in J$, $\b x_{\bsl I J}$ the value of $X_h,\:h\in \bs I J$ and $\b x_{\bsl M I}$ the values of $X_h,\:h\in\bs M I$. We will also write $\b 0_{\bsl I J}$ to state that $X_h=0,\: \forall h\in \bs I J$.
\subsection{Marginal and conditional log-linear interactions}
Though there are many different ways of coding marginal log-linear
parameters, parameters defined by different codings are linearly
related; thus there is no loss of generality in using the
{\it reference category} coding, where comparisons are with respect to the category taken as reference, usually the first.
\begin{definition}
A reference category log-linear interaction $I$ within $M$ is defined by the
following expression
\begin{equation}
\eta_{I;M}(\b x_I\mid \b x_{\bsl M I})= \sum_{J\subseteq
I}(-1)^{\mid \bsl I J\mid} \log p_M(\b x_J,\b 0_{\bsl I J},\b
x_{\bsl M I}),\label{eq:mlli}
\end{equation}
where, $\forall  i\in I, x_i>0$.
\end{definition}
\begin{example}
The logit of $X_i$ at $x_i$ computed within $M$ is
$$
\eta_{i;M}(x_i\mid \b x_{\bsl M i})=\log p_M(x_i,\b x_{\bsl M i})-
\log p_M(0_i,\b x_{\bsl M i}),\quad x_i >0
$$
and the log-odds ratio for $X_i=x_i,\:X_j=x_j$ is
\begin{eqnarray*}
\eta_{H;M}(x_i,x_j\mid \b x_{\bsl M H})&=&\log p_M(x_i,x_j,\b
x_{\bsl M H})- \log p_M(0_i,x_j,\b x_{\bsl M H})\\
&-& \log p_M(x_i,0_j,\b x_{\bsl M H})+ \log p_M(0_i,0_j,\b x_{\bsl M
H})
\end{eqnarray*}
where $H=i\cup j$.
\end{example}
It may be easily verified that, given $h\in \bs M I$ and $H$ =
$I\cup h$, (\ref{eq:mlli}) implies the following recursive relation
\begin{equation}
\eta_{H;M}(\b x_H\mid \b x_{\bsl M H})= \eta_{I;M}(\b x_I\mid x_h,\b
x_{\bsl M H})- \eta_{I;M}(\b x_I\mid 0_h,\b x_{\bsl M H})
,\label{eq:recurs}
\end{equation}
this indicates that interactions of higher order may be constructed
by a sequence of first order differences starting from logits.

Whenever $\bs M I$ is not empty, marginal log-linear interactions
depend on the value of the remaining variables. Because
(\ref{eq:mlli}) is a contrast of logarithms of marginal
probabilities, it can be easily verified that $\eta_{I;M}(\b x_I\mid
\b x_{\bsl M I})$ is the log-linear interaction $I$ in the marginal
distribution $M$ conditionally on $X_h=x_h\:\forall h\in \bs M I$.
Clearly, within the full collection of marginal log-linear
interaction parameters conditional on the configurations of the
remaining variables, there is a substantial amount of redundancy.
Below we show that these parameters are linearly related and that
they can all be written in terms of the subset where the
conditioning variables are all fixed at their reference category; this
subset contains non redundant elements.

For a given $\eta_{I;M}(\b x_I\mid \b x_{\bsl M I})$ let $h\in \bs M
I$, then (\ref{eq:recurs}) may be used to obtain
$$
\eta_{I;M}(\b x_I\mid \b x_{\bsl M I})= \eta_{I;M}(\b x_I\mid
 0_h,\b x_{\bsl M H})+ \eta_{H;M}(\b x_I,x_h\mid \b x_{\bsl M H}).
$$
Repeated use of the relation above
leads to the following expansion
\begin{equation}
\eta_{I;M}(\b x_I\mid \b x_{\bsl M I})= \sum_{I\subseteq H\subseteq
M} \eta_{H;M}(\b x_H\mid \b 0_{\bsl M H}).\label{eq:linexp}
\end{equation}
The above equation shows that any marginal log-linear interaction may be written as a linear function of all possible higher order interactions conditional to the initial category of the remaining variables within the given marginal. For simplicity, in the following, we write $\eta_{I;M}(\b x_I)$ as a shorthand for $\eta_{I;M}(\b x_I\mid \b 0_{\bsl M I})$. An alternative way of removing conditioning variables, which has been applied to interactions defined as contrasts of averages of logarithms of probabilities, but could be applied to any type of interactions, is to average across the set of all possible configurations of the conditioning variables $\b x_{\bsl M I}$. The log linear interactions used by \citet{BerRud:02}, among others, are defined in this
way; Lemma \ref{BRCF} in the Appendix shows that these interactions are linear functions of all the interactions $\eta_{H;M}(\b x_J)$ for $H\supseteq I$.
\begin{example}
Suppose that $M$ = $I\cup h\cup k$, then
$$
\eta_{I;M}(\b x_I\mid x_h,x_k)=\eta_{I;M}(\b x_I)+\eta_{I\cup
h;M}(\b x_I,x_h)+\eta_{I\cup k;M}(\b x_i,x_k)+\eta_{I\cup h\cup
k;M}(\b x_I,x_h,x_k).
$$
\end{example}

For any $I\in \cg P(M)$, it is convenient to arrange the log-linear interactions $\eta_{I;M}(\b x_I)$ into the vector $\b \eta(I,M)$ with elements in lexicographic order of $\b x_I$; this vector may be written as
\begin{equation}
\b \eta(I,M) = \b C(I,M) \log \b p(M),
 \label{loglin}
\end{equation}
where $\b C(I,M)$ = $\bigotimes_{j=1}^d \b C_j$ and $\b C_j$ = $\left(-\b 1_{r_j}\: \b I_{r_j}\right)$ if $j\in I$ and
$\b C_j$ = $\left(1,\b 0_{r_j}\tr\right)$ otherwise.
Let also $\b\eta(M)$ = $\b C(M) \log \b p(M)$ denote the vector obtained by stacking the $\b \eta(I,M)$ components one below the other in lexicographic order relative to $I\in\cg P(M)$. It is well known that under multinomial sampling, $\b \eta(M)$ { constitutes} a vector of variation independent canonical parameters for $\b p(M)$. Let $\b G(I,M)$ = $\bigotimes_{j=1}^d \b G_j$, where $\b G_j$ is an identity matrix of order $r_j+1$ without the first columns if $j\in I$ and $\b 1_{r_j+1}$ otherwise. Let $\b G(M)$ be the matrix whose columns are given by the $\b G(I,M)$ matrices arranged one aside the other in lexicographic order. It is easily verified that $\b G(M)$ is the right inverse of $\b C(M)$; this implies  the reconstruction formula
{
\begin{equation}
\log \b p(M) = \b G(M)\b \eta(M) -\b 1 \log\{\b 1\tr \exp[\b G(M)\b
\eta(M)]\}.
\label{reconstruct}
\end{equation}
}
\subsection{The mixed parameterization}
Within the distribution in $M$, the vector of {\em mean} parameters
$\b\mu_{\cgl P(M)}$  \cite[p. 121]{BarNi:79} is the expected value
of the sufficient statistics for $\b\eta(M)$ in a sample of size 1
and equals
$$
\b\mu_{\cgl P(M)} = \bt G(M)\b p(M);
$$
there is a diffeomorphism between $\b\mu_{\cgl P(M)}$ and
$\b\eta(M)$  \citep[p. 121]{BarNi:79}.
Because each block of rows $\b C(I,M)$ in $\b C(M)$ corresponds to a block of columns $\b G(I,M)$ in $\b G(M)$,
we may define $\b\mu(I)$ = $G(I,M)\tr \b p(M)$ to be the collection of mean parameters for a given interaction. It is worth noting that, though mean parameters, like canonical parameters, are associated to interactions $I\in \cg P(M)$, $\b \mu(I)$ may be defined in any marginal such that $I\subseteq M$. Having coded the canonical parameters as contrasts with respect to the initial category, the corresponding mean parameters are simply marginal probabilities.

We recall a definition and a few results which are relevant in the following.
\begin{definition}
For an arbitrary margin $M$, let $(\cg U,\:\cg V)$ be a partition of
$\cg P(M)$; the pair of
vectors $[\b\eta_{\cgl U,M}, \b \mu_{\cgl V}]$, where $\b\eta_{\cgl
U,M}$ = $(\b \eta(I,M), I \in \cg U)$ is composed of canonical
parameters, and $\b \mu_{\cgl V}$ = $(\b \mu(I), I \in \cg V)$ is
composed of mean parameters, constitute a {\em mixed
parameterization} of the marginal distribution $\b p(M)$.
\end{definition}
 In the following, to be short, we will often refer to the log-linear parameters $\b\eta_{\cgl
U,M}$ = $(\b \eta(I,M), I \in \cg U)$ as {\it log-linear interactions in $\cg U$ } or {\it collection $\cg U$ of log-linear parameters}.
\begin{lemma}
For any mixed parameterization, there is a diffeomorphism between
the vector of mean parameters $\b \mu_{\cgl P(M)}$ and the pair of
vectors $[\b\eta_{\cgl U,M}, \b \mu_{\cgl V}]$; in addition, the
two components are variation independent.
\end{lemma}
\begin{proof}
See \citep[][p. 121-122]{BarNi:79}
\end{proof}
The numerical algorithm for reconstructing $\b p(M)$ from
$[\b\eta_{\cgl U,M}, \b \mu_{\cgl V}]$ given by \citet{Forcina:11}
is a faster alternative to the usual IPF algorithm.

The mixed parameterization is a powerful tool for reconstructing a joint distribution from marginal log-linear parameters because one can process one marginal distribution at a time by combining the log-linear parameters defined within that distributions with the mean parameters, or, equivalently, marginal probabilities, available from marginal distributions reconstructed in previous steps. As long as these two sets of interactions are a partition of $\cg P(M)$, the basic argument used by \citet{BaCoFo:07} implies that any model defined by linear constraints on the marginal log-linear parameters constitutes a curved exponential family and thus is smooth.
\section{The LM reconstruction algorithm}
In this section we investigate the properties of conditional independence models which require to impose non trivial constraints on the same log-linear interactions defined in two or more marginal distributions. We may suppose, without loss of generality, that the marginals of interest are arranged in non decreasing order and that they will be processed one at a time, starting from the first one. In this way, at each step in the reconstruction of the joint distribution from its marginal log-linear parameters, we need only be concerned with the marginal at hand and examine whether, by use of the mixed parameterization, we may combine the mean parameters from previous marginals with the log-linear parameters which are either available or need to be constrained in the marginal under consideration. An algorithm for doing this is presented and its convergence properties investigated.
\subsection{Motivating examples}
We now present a set of examples which will highlight different features of the kind of models we are going to consider. Each model is made of two parts: (i) a list of conditional independencies
 which have  been accommodated, somehow, in previous marginals (ii) an additional conditional independence  to be imposed in the current marginal $M$. We start with a couple of elementary models:
\begin{example}\label{ex:1} 
Suppose that, having assumed that $1\ci 2\mid 3$ in the marginal 123, in the marginal $M=1234$ we want also $1\ci 2 \mid (3,4)$. Here we need to constrain again the $\{12,123\}$ interactions; in the binary case, \citet{RobEva} has shown that the model has singularities.
\end{example}
\begin{example}\label{ex:2}  
Suppose that, having assumed that $1\ci (2,4)$ in the marginal 124, we want also $2\ci 4 \mid (1,3)$. Here, in addition to the $124$ interaction which has already been constrained in $124$, we need to constrain $24$ which was defined in the previous marginal; in the binary case, \citet{Drton:DCG} has shown that the model has singularities.
\end{example}
The nest example is a little more complex:
\begin{example} \label{ex:6} 
Having assumed that $1\ci 2\mid 3$ and $1\ci 3\mid 4$ we also want $1 \ci (2,3)\mid (4,5)$; here the list of interactions to be constrained again is given by $\{12, 123, 13, 134\}$.
\end{example}
The following examples are different because the collection of interactions that have already been defined in previous marginal is too large to be redefined again in $M$:
\begin{example}\label{ex:5}  
Suppose that, having set $1\ci 2\mid (3,4)$ and $1\ci 2\mid (3,5)$ we also want $1\ci 2\mid (3,4,5)$; here the collection of interactions that have already been defined and that have to be constrained again is $\{12,123,124,125,1234,1235\}$.
\end{example}
\begin{example} \label{ex:7}   
Suppose that, having set $1\ci 2\mid (3,4)$, $1\ci 2\mid (3,5)$ and $1\ci 2\mid (4,5)$ we also want $1\ci 2\mid (3,4,5)$, here all the interactions in the ascending class from $12$ to $M=12345$, except $M$ itself, have to be constrained again.
\end{example}
\subsection{Setting up the framework}
Let $M$ denote the current marginal, $\cg V$ the collection of interactions defined in previous marginals which belong to $\cg P(M)$ and $\cg L$ = $\cg P(M)\bks \cg V$. Let also $\cg A$ be the collection of interactions to be constrained in $M$ according to the last conditional independence statement; whenever $\cg V\cap \cg A\ne \emptyset$, we are trying to constrain again the corresponding log-linear interaction.  Though we would like to redefine and constrain in $M$ all the interactions in $\cg V\cap\cg A$, we shall see that this is not always possible; denote by $\cg I\subseteq (\cg V\cap \cg A)$ the actual collection which we
redefine in $M$ and $\cg R=\cg V\bks \cg I$ the remaining interactions.

Because the mean parameters in $\cg I\cup \cg R$ together with the log-linear parameters in $\cg L$ constitute a mixed parameterization of $\b p(M)$, these parameters determine uniquely the value of the log-linear parameters in $\cg I$ to be redefined within $M$; thus they cannot be constrained again,  unless we remove from $\cg L$ a collection, say $\cg H$, of log-linear interactions with exactly the same number of parameters as the collection $\cg I$; below we investigate whether such an atypical parameterization may provide a smooth mapping. We shall see that the two sets $\cg I,\cg H$ must be chosen carefully and satisfy a set of conditions which establish a close relation between them.
\begin{example}
Consider again example \ref{ex:5}, here $\cg V$ = $\cg P(1234) \cup \cg P(1235)$, $\cg A$ = $\{12, 123,124,$ $125,1234, 1235,1245,12345\}$ and $\cg A\cap \cg V$ = $\{12,123,124,125,1234,1235\}$; as we shall see, not all the elements of this collection can be redefined in $M$, the most we can achieve is to set $\cg I$ = $\{12,123\}$ and $\cg H =\{1245,12345\}$ where $X_4$ and $X_5$ are fixed to a given category.
\end{example}
 \begin{example}
In example \ref{ex:2}, $\cg V = \cg P(124)$, $\cg L =\cg P(1234)\bks \cg P(124)$; suppose we set $\cg I=\{24,124\}$ and $\cg H=\{234,1234\}$ where $X_3$ is fixed to a given category; it can be easily checked that $\cg H$ indexes the same number of parameters as $\cg I$.
\end{example}
\subsection{Description of the algorithm}
The problem, when reconstructing the distribution in $M$, is how to combine the mean parameter $\b\mu_{\cgl I}$, available from previous marginals with the log-linear parameters $\b\eta_{\cgl I;M}$ defined again in the present marginal. Recall that the mixed parameterization require that mean parameters and log-linear interactions must refer to two complementary sets whose union is $\cg P(M)$. The idea is to remove from the log-linear parameters $\cg L$, to be defined in $M$, the subset $\cg H$ with the same number of parameters as the elements of $\cg I$. The algorithm that we describe below can handle such a context and  the issue will be to determine under which conditions such an algorithm may converge; if it does, then it can be shown that the model is smooth. The algorithm for reconstructing the marginal distribution in $M$ is made of two steps and require starting values for $\b\eta_{\cgl H;M}$:
\begin{description}
\item{M-step}
given the latest guess for the log-linear parameters $\b\eta_{\cgl
H;M}$, an updated estimate for the vector of mean parameters
$\b\mu_{\cgl H}$ may be computed by a mixed parameterization with
mean parameters indexed by the collection of interactions $\cg R\cup \cg I$ and log-linear parameters indexed by $\cg L$;
\item{L-step}
given the latest guess for the vector of mean parameters $\b\mu_{\cgl H}$, an updated estimate for the vector of log-linear parameters $\b\eta_{\cgl H;M}$ may be computed by a mixed parameterization with mean parameters indexed by $\cg R\cup \cg H$ and log-linear parameters indexed by $\cg I\cup (\cg L\bks H)$.
\end{description}

In order to examine the properties of the LM algorithm, we need to determine how changes in the input value of $\b\eta_{\cgl H;M}$ in the M step affects the output value produced in the L step. For this purpose, we recall results concerning the derivatives of certain components of the mixed parameterization relative to others which are relevant here. In the following write $\b \pi$ as a shorthand for $\b p(M)$, let $\b D_{\bl\pi}$ = $\diag(\b\pi)$ and let $\b \Omega=\b D_{\bl\pi} -\b \pi \b \pi \tr$
denote the derivative of $\b \pi$ with respect to $\b \eta(M)\tr$.
\begin{lemma}
$$
\b F(M) = \frac{\partial\b \mu_{\cgl P(M)}}{\partial\b\eta(M)\tr} =
\b G(M)\tr \b\Omega(M) \b G(M),
$$
is the covariance matrix of a collection of distinct binary
variables determined by the columns of $\b G(M)$ and thus is positive definite.
\end{lemma}
\begin{proof}
See \citet{Forcina:11}.
\end{proof}
Any two subsets of interactions $\cg H,\: \cg K\subseteq \cg P(M)$ determine two sub-collections of binary random variables and a block in the covariance matrix $\b F(M)$. In the following we omit reference to the marginal $M$ when it is obvious from the context and write
$$
\b F_{\cgl H\cgl K} = \b G_{\cgl H}\tr \b\Omega \b G_{\cgl K}.
$$
\begin{lemma}
In the M-step, where $\cg H$ is part of the log-linear parameter
$$
\frac{\partial\b\mu_{\cgl H}}{\partial\b\eta_{\cgl H;M}\tr}=\b B =
\b F_{\cgl H\cgl H}-\b F_{\cgl H \cgl V}\b F_{\cgl V \cgl V}^{-1}\b
F_{ \cgl V \cgl H};
$$
in the L-step, where $\cg H$ is part of the mean parameter
$$
\frac{\partial \b\eta_{\cgl H;M}}{\partial \b\mu_{\cgl H}}=\b A^{-1} =
\left(\b F_{(\cgl R\cup \cgl H)(\cgl R\cup \cgl H)}^{-1}\right)_{\cgl H\cgl H} = \left( \b
F_{\cgl H\cgl H}-\b F_{\cgl H\cgl R}\b F_{\cgl R\cgl R}^{-1}\b
F_{\cgl R \cgl H}\right)^{-1},
$$
where we have used the formula for the inverse of a partitioned
matrix.
\end{lemma}
\begin{proof}
the result follows from Lemma 4 in \citet{Forcina:11}.
\end{proof}
A full step of the LM algorithm may be seen as a fixed point
function which, given a guess value of $\b\eta_{\cgl H;M}$,
produces an updated estimate of the same vector. A sufficient
condition for an algorithm to be a contraction \cite[see for
example][]{Agarw:2001}, a property which implies that it converges to a unique solution, is that the jacobian of a full LM step has spectral radius (maximum absolute eigenvalue) strictly smaller than 1. Let $\b J$ = $\b A^{-1}\b B$ be the jacobian of this mapping; let also $\b Q_{\cgl I\cgl H\mid\cgl R}$ =  $\b F_{\cgl I\cgl H}-\b F_{\cgl I\cgl R}\b F_{\cgl R\cgl R}^{-1}\b F_{\cgl R \cgl H}$. An upper bound for the spectral radius of $\b J$ is determined in the following lemma.
\begin{lemma}\label{le:spectr}
The spectral radius of $\b J$ is always less than 1 except when $\b Q_{\cgl I\cgl H\mid\cgl R}$ is not of full rank.
\end{lemma}
\begin{proof} { See the Appendix.}
\end{proof}
The main result of this section is contained in the following Theorem and concerns the properties of the mapping from  $\b \xi$ = $(\b\eta_{\cgl L\cup \cgl I\bks \cgl H, M}, \b\mu_{\cgl V} )$ to $\b \pi$, under the assumption that the elements of $\b\xi$ are compatible, that is there is at least a $\b\pi$ with the parameters specified by $\b\xi$. The result depends on the spectral radius of the jacobian matrix $\b J$ defined above.
\begin{theorem}
\label{th:conv}
Under the assumption that the elements of $\b\xi$ are compatible, when $\b Q_{\cgl I\cgl H\mid\cgl R}$ is of full rank, the mapping from $(\b\eta_{\cgl L\cup \cgl I\bks \cgl H, M}, \b\mu_{\cgl V} )$ to $\b \pi$ is one to one and smooth. In the special case when $\b Q_{\cgl I\cgl H\mid\cgl R}=\b 0$, so that $\b J$ is an identity matrix, the mapping is not one to one. When $\b Q_{\cgl I\cgl H\mid\cgl R}$ is singular but different from a null matrix, the algorithm does not converge and nothing can be said about the smoothness of the mapping.
\end{theorem}
\begin{proof}
Consider the sequence of vectors produced by the LM algorithm: $\b\eta_{\cgl H;M}^{(0)}, \b\eta_{\cgl H;M}^{(1)}, \dots $, where $\b\eta_{\cgl H;M}^{(0)}$ is the starting value and $\b\eta_{\cgl H;M}^{(s)}$ is the output of one step of the LM algorithm when we use $\b\eta_{\cgl H;M}^{(s-1)}$ as input; because we have assumed that there is at least a compatible solution inside the parameter space, \cite[][Theorem 1.1]{Agarw:2001} implies that, if the spectral radius of $\b J$ is strictly less than one, the sequence converges to a unique solution. At convergence the argument in \cite[Theorem 1]{BaCoFo:07} can be applied to show that the mapping is a diffeomorphism. In the special case when the jacobian matrix $\b J$ is an identity matrix, $\b\eta_{\cgl H;M}^{(0)}$ = $\b\eta_{\cgl H;M}^{(1)}$, so the algorithm converges in one step, irrespective of the starting value. This implies that, if $\b\pi^{(0)}$ is the probability vector corresponding to $\b\eta_{\cgl H;M}^{(0)}$ there is a whole neighbourhood of $\b\pi^{(0)}$ whose points share exactly the same vector $\b\xi$ of mean and log-linear parameters.
\end{proof}
\begin{remark}
According to Theorem \ref{th:conv}, a model may be smooth even if the log-linear interactions in $\cg I$ are defined and constrained in two different marginals. This is apparently in conflict with the result of \citep[Theorem 3]{BerRud:02} which says that the jacobian obtained by differentiating the same log-linear interaction $I$ defined in two different marginals, say $M_1,\:M_2$, with respect to $\b p$, is singular for the uniform distribution, a condition which is necessary (but not sufficient) for a model to have singularities. However, when the set $M\bks I$ is not empty, the log-linear interactions defined by \citet{BerRud:02} are constructed by averaging conditional interactions across all possible configurations of the conditioning variables. As mentioned in section 2.1, the results of Lemma \ref{BRCF}  in the Appendix imply that any constraint on one of their log-linear interactions
is equivalent to a linear constraint on the whole ascending class of our interactions with minimal element $I$ and maximal element $M$. Hence the LM algorithm is not directly applicable to interactions defined in that way.
\end{remark}
\subsection{Convergence of the algorithm}
Below we derive a more convenient expression for $\b Q_{\cgl I\cgl H\mid\cgl R}$ and show that the matrix is non singular under complete independence, if the set $\cg H$ satisfies certain conditions. We also determine conditions under which $\b Q_{\cgl I\cgl H\mid\cgl R}$ is singular or null. Finally, we discuss the singularity of the same matrix when $\b p(M)$ is unrestricted.

Let $\b P_{\emptyset} =\b 1 \b\pi\tr$ be the projector, according to the metric defined by the matrix $\b D_{\bl\pi}$, on the space spanned by the vector $\b 1$. By simple algebra, it can be shown that:
\begin{eqnarray*}
\b F_{\cgl I\cgl H} &=& \b G_{\cgl I}\tr \b \Omega \b G_{\cg H} =  \b G_{\cgl I}\tr(\b I - \b P_{\emptyset})\tr \b D_{\bl\pi} (\b
I - \b P_{\emptyset})\b G_{\cgl H}.
\end{eqnarray*}
From the previous result, it follows that:
$$
\b Q_{\cgl I \cgl H \mid \cgl R} = \b G_{\cgl I}\tr \b D_{\bl\pi} (\b I -
\b P_{\cgl{\ovl R}})(\b I - \b P_{\emptyset})\b G_{\cgl H},
$$
where
$$\b P_{\cgl{\ovl R}}=(\b I - \b P_{\emptyset})\b G_{\cgl R}\b F_{\cgl R\cgl R}^{-1}
\b G_{\cgl R}\tr
(\b I - \b P_{\emptyset})\tr \b D_{\bl\pi}$$
is the projector, according to the metric defined by the matrix $\b D_{\bl\pi}$, on the space spanned by the columns of $\cg (\b I-\b P_\emptyset)\b G_{\cgl R}$.

Let $\cg S(\b X)$ denote the space spanned by the columns of $\b X$, for every $a\subseteq M$ let also $\b X_{a}= \bigotimes_{j \in
M} \b X_j$, where  $\b X_j =\b I_j$ if $j \in a$ and $\b X_j =\b 1_j$  otherwise, and let $\b P_a$ = $\b X_{a}(\b X_a\tr \b D_{\bl\pi}\b X_{a})^{-1}\b X_a\tr \b D_{\bl\pi}$ be the projection matrix onto $\cg S(\b X_a)$. Let $\b X_{\cgl I \cup \cgl R}$ be the matrix made by the columns of $\b X_a\: \forall a\in \cg I\cup \cg R$
and $\b P_{\cgl I \cup \cgl R}$ the projection onto $\cg S(\b X_{\cgl I \cup \cgl R})$.
Because $\cg S(\b G_{\cgl R})$ and $\cg S(\b 1)$ belong to $\cg S(\b X_{\cgl I \cup \cgl R})$ the projection matrix $\b P_{\cgl I \cup \cgl R}$ commutes with both $\b P_{\cgl{\ovl R}}$ and $\b P_{\emptyset}$; in addition, by using the identity  $\b P_{\cgl I \cup \cgl R}\b G_{\cgl I}=\b G_{\cgl I}$ it follows that
\begin{equation}
\b Q_{\cgl I \cgl H \mid \cgl R} = \b G_{\cgl I}\tr \b D_{\bl\pi} (\b I -
\b P_{\cgl{\ovl R}})(\b I - \b P_{\emptyset})\b P_{\cgl I \cup \cgl R}\b G_{\cgl H}.
\label{eq:Qform}
\end{equation}
\subsubsection{The case of complete independence}
\begin{lemma}\label{le:noH}
Under complete independence of the variables in $M$, (i) if $\cg H$ contains an interaction $v\not\in\cg A$, the corresponding columns in $\b Q_{\cgl I \cgl H \mid \cgl R}$ are null, (ii)
if $\cg H$ contains an interaction $v$ where at least one of the variables in $v$ is not binary and not contained in any element of $\cg V$, $\b Q_{\cgl I \cgl H \mid \cgl R}$ has a block of columns which is not of full rank.
\end{lemma}
\begin{proof}
{See the Appendix}
\end{proof}
Lemma \ref{le:noH} suggests two necessary conditions for $\b Q_{\cgl I \cgl H \mid \cgl R}$ to be non singular: $\cg H$ cannot contain interactions not in $\cg A$ and all the non binary variables involved in the class of interactions $\cg H$ and not present in the class $\cg I\cup\cg R$, must be fixed to a single category different from the reference category. This implies that only a limited number of higher order interactions in $M$ can be used as a replacement for those in $\cg I$.
The definition below provides a set of conditions for $\cg H$ which will be shown to be sufficient.
Let $\cg K = \{m_1,\dots , m_r\}$ be the family of the maximal sets of $\cg I\cup\cg R$; for $t \in \cg I$, let $K(t)=\{m: m \in \cg K, t \subseteq m\}$ be the family of sets $m\in \cg K$ that contain $t$, $\cg K(t,h)$ be the family of the sets $\cg G$,  $\cg G \in \cg P[\cg K(t)] $, such that $h\cap \bigcap_{m_j\in\cgl G} m_j = \emptyset$ and $\bar{\cg K}(t,h)=\cg P(\cg K)\bks\cg K(t,h)$.
\begin{definition} \label{conf}
A set $\cg H$ is a {\it valid replacement} for a given $\cg I$ if it satisfies the following conditions:
\begin{itemize}
\item[(i)] there is a one to one correspondence between the elements of $\cg I$ and $\cg H$ such that, for each $t\in \cg I$, there is a $v=t \cup h \in\cg H$, $t \cap h= \emptyset$, where the variables in $h$ are fixed to a given category different from the reference category;
\item[(ii)] $\sum_{\cgl G\in \cgl K(t,h)}\: (-1)^{\mid \cgl G\mid} \ne 0$;
\item[(iii)] there exists a complete ordering "$\prec$" in $\cg I$, coherent with the partial ordering  of set inclusion, such that, for every $t \cup h \in \cg H$, $\cg G\in \bar{\cg K}(t,h)$ and $s=\left(\bigcap_{m_j\in \cgl G}m_j \right) \cap (t\cup h)$  either $s \in \cg R$, or $s\in \cg I$ and  $s \prec t$.
\end{itemize}
\end{definition}
To clarify these notions, we discuss a few examples where we write $(t,h)$ as a shorthand for $t \cup h$ if $t \cup h \in \cg H$.
\begin{example}
In example \ref{ex:1} with $\cg I=\{12,123\}$ and $\cg H=\{(12,4),(123,4)\}$, all conditions are trivially satisfied with $\cg K=\{123\}$; this is also the only element of $\cg K(t,h),\:\forall t,h$, the same happens in example \ref{ex:2}. In example \ref{ex:6} with the ordered set  $\cg I=\{12,13, 123,134\}$ and $\cg H=\{(12,5),(13,5),(123,5),(134,5)\}$ condition (i) is clearly satisfied. In this case we have: $\cg K=\{123,134\}$,  $\cg K(134,5)=\{134\}$, $\cg K(123,5)=\{123\}$, $\cg K(13,5)=\{123,134,\{123,134\}\}$ and $\cg K(12,5)=\{123\}$ and condition (ii) is satisfied by these sets.
For $\cg{\bar K}(134,5)=\{123,\{123,134\}\}$ condition (iii) is satisfied with $s=13$. The set $\cg{\bar K}(123,5)=\{134,\{123,134\}\}$ satisfies (iii) because s=13. In the case of $\cg{\bar K}(12,5)=\{134,\{123,134\}\}$ (iii) holds because $s=1$. The family  $\cg{\bar K}(13,5)$ is empty and so in this case condition (iii) is void.
\end{example}

\begin{example}
\label{ex:8}
Having assumed $1\ci 2\mid (3,4),\: 1\ci 2\mid (3,5),\:1\ci 2\mid (3,6)$, we also want $1\ci 2\mid (4,5,6)$. It can be verified that $\cg H=\{(12,56),(124,56)\}$ is a valid replacement for $\cg I=\{12,124\}$; here $\cg K=\{124,125,126\}$ and $\cg K(124,56)=\{124\}$. It is easy to see that (i) and (ii) are satisfied. Condition (iii) holds in the case of  $\bar{\cg K}(124,56)$ = $\{125,126,\{124,125\}, \{124,126\},\{125,126\}, \{124, 125,126\}\}$, because apart from the first two elements which produce sets $s$  that belong to $\cg R$, all the others produce  $s=12$. A similar remark holds in that case of $\cg K(124,56)$. However, if we set $\cg I=\{12,124,125,126\}$, though $\cg H=\{(12,56),(124,56),$ $(125,4),(126,4)\}$ satisfies conditions (i) and (ii), (iii) does not hold.
\end{example}
We now give an instance where condition (ii) is not satisfied.
\begin{example}
\label{ex:14}
Suppose that $1\ci 2\mid (3,4)$, $1\ci 2\mid (3,5)$ and finally $1\ci 2\mid (3,4,5,6)$. Though the best choice would be to set $\cg I=\{12,123,124,125,1234,1235\}$, if we set $\cg I=\{12,123\}$ and $\cg H=\{(12,46),(123,46)\}$, condition (ii) is not satisfied. 
\end{example}
\begin{lemma} \label{le:ex}
A pair $\cg I$, $\cg H$, where $\cg H$ is a valid replacement, always exists; for instance, take $\cg I =\{t\}$, where $t$ is one of the minimal elements of $\cg A$, and $\cg H = \{t \cup h\}$, where $h$ contains all the variables that belong to at most one element of $\cg K(t)$ when $\cg K(t)$ is not a singleton and by the variables that do not belong to the unique element of $\cg K(t)$ otherwise.
\end{lemma}
\begin{proof} {  See the Appendix }
\end{proof}
\begin{example}
In example \ref{ex:6},  the minimal element  $t$ of $\cg A$ can be 12 or 13. If $t=12$ then $\cg K(t)= \{123\}$ is a singleton. In this case h=45 and $\cg H$ contains only $(12,45)$, the family $\cg K(t,h)$ contains only $\{123\}$. Instead, if we set $t=13$, $\cg K(t)= \{123, 134\}$ and we must set $h=45$, thus $\cg K(t,h)$ contains only the set $\cg G=\{123, 134\}$.
In example \ref{ex:14}, $\cg K =\{1234,1235\}$, with $t=12$ we must set $h=456$ and $\{1234,1235\}$ is the only set in $K(t,h)$.
In example \ref{ex:8} with $t=12$, $\cg K(t)=\{1234,1235,1236\}$, thus we must set $h=456$; here $\cg K(t,h)$ has 3 elements of size 2 and 1 element of size 3 and the sum in condition (ii) of Definition 1 is -2.
\end{example}

Let $\b G_{t,h}(\b j_h)$ be the sub-matrix of $\b G_{t\cup h}$ where variables in $h$ are fixed to $\b j_h$.
\begin{lemma}
\label{le:prod}
Under complete independence:\\
a) $\cg S(\b P_a\b G_{t,h}(\b j_h))\subseteq \cg S(G_r)$, where $r=a\cap (t\cup h)$,\\
b) if $t\subseteq a$ and $a\cap h=\emptyset$,  $\b P_a\b G_{t,h}(\b j_h)$ = $\b G_t { P(\b x_h=\b j_h)}$.
\label{le:tabel}
\end{lemma}
\begin{proof}
{See the Appendix.}
\end{proof}
\begin{theorem}
\label{th:ind}
If $\cg H$ is a valid replacement for $\cg I$, under complete independence, $\b Q_{\cgl I \cgl H \mid \cgl R}$ is non singular.
\end{theorem}
\begin{proof}
Because under complete independence the projectors $\b P_a, \: a \subseteq M$ commute, we can write $\b P_{\cgl I \cup \cgl R}$ = $\sum_{\cgl G \in \cgl P(\cgl K)}(-1)^{1+\mid \cgl G\mid }\prod_{m \in \cg G}\b P_m$,  it follows that:
$$
\b P_{\cgl I \cup \cgl R}\b G_{t,h}(\b j_h) = \sum_{\cgl G \in \cgl  K(t,h)}(-1)^{ 1+\mid \cgl G\mid} \prod_{m \in \cgl G}\b P_m \b G_{t,h}(\b j_h) + \sum_{\cgl G \in \bar{\cgl K}(t,h)}(-1)^{1+\mid\cgl G\mid} \prod_{m \in \cgl G}\b P_m \b G_{t,h}(\b j_h).
$$
Condition (ii) of definition \ref{conf} and b) in lemma \ref{le:tabel} imply that the first sum is $k\b G_t$, with  $k \neq 0$. Condition (iii) of definition \ref{conf} and a) in lemma \ref{le:tabel} imply that,
when an element in the second sum, say $\b U$, is such that $\cg S(\b U)\subseteq \cg S(\b G_{\cgl R})$, when we left multiply by $(\b I - \b P_{\cgl{\bar R}})(\b I-\b P_\emptyset)$, we get a null matrix because $(\b I - \b P_{\cgl{\bar R}})$ projects onto the space orthogonal to $(\b I-\b P_\emptyset)\b G_{\cgl R}$. { In all other cases there exists a non null matrix  $\b A_{s,t}$, $s\prec t$}, such that:
$$
(\b I - \b P_{\cgl R})(-1)^{1+\mid\cgl G\mid} \prod_{m \in \cgl G}\b P_m \b G_{t,h}(\b j_h)=(\b I - \b P_{\cgl R}) \b G_s \b A_{s,t}.
$$
The matrix $\b G_{\cgl H}$ is made of blocks of columns of the form $\b G_{t,h}(\b j_h)$ and we may assume,{ without loss of generality,  that these blocks are in the same order as the elements of $\cg I$ specified }in condition (iii), then it follows that
$$
\b Q_{\cgl I \cgl H \mid \cgl R} = \b Q_{\cgl I \cgl I \mid \cgl R} \b A,
$$
where the matrix $\b A$ has blocks $\b A_{s,t}, \quad s,t \in \cg I$ such that $\b A_{s,t}=\b 0$ if $t \prec s$ and,
because of condition (ii) of definition \ref{conf}, the diagonal blocks $\b A_{t,t}$, are proportional to an identity matrix, thus $\b A$ is lower triangular and non singular. The result follows because both matrices in the product above are non singular.
\end{proof}
\begin{example}
The choice of $\cg I,\:\cg H$ in the second part of example \ref{ex:8} does not satisfies (ii) still, numerical simulations indicate that $\b Q_{\cgl I \cgl H \mid \cgl R}$ is non singular. This  exemplifies  that the conditions of being adequate for replacement are only sufficient.
\end{example}
\begin{remark}
Theorem 6 of \citet{RoLuLa} implies that, in the binary case, the model defined by $a\ci b$ and $a\ci b\mid c$, with all the elements of $c$ equal 0, is smooth. If we set $\cg I$ = $\cg P(a\cup b)\setminus(\cg P(a)\cup \cg P(b))$ and $\cg H$ = $\{t,h: t\in \cg I,\: h=c(\b j_c)\}$ with $\b j_c=\b 1$, our Theorem \ref{th:ind} implies that, under independence, the model $a\ci b$ and $a\ci b\mid c$, except when all the elements of $c$ are equal 1, is smooth.
\end{remark}
\subsubsection{The general case}
Unfortunately, the main arguments used above depend crucially on the assumption of complete independence. For discrete data, all the models which have been shown to be non smooth, have a singular locus which is a subset of that defined by complete independence; instead, for gaussian models \citep[Example 4.4]{DrtXiao} indicates that a non smooth model may have a singular locus not contained in the model of complete independence. { It is also interesting to note that numerical evaluations of $\b Q_{\cgl I \cgl H \mid \cgl R}$ outside the space of complete independence, indicate that it is of full rank even if $\cg H$ does not satisfies the conditions of Definition 3.}

Taking into account all of the above, the extensive simulations which we have performed seem to support the conjecture that all the models obtained by replacing the interactions in $\cg I$ with an adequate replacement in $\cg H$ are indeed smooth everywhere in the parameter space. Though very unlikely, we cannot rule out the possibility that the models above may be singular on points of the parameter space which are outside the subspace defined by the model of complete independence.
However, the expression for $\b Q_{\cgl I \cgl H \mid \cgl R}$ is very easy to compute and the software in {\sc MatLab} and R which we provide as supplementary material, may be used for quick numerical checks.
\section{Context specific conditional independence models}
We have seen that a non smooth model may be transformed into a smooth one by omitting a class of log-linear interactions $\cg H$ to make place for other interactions defined in a previous marginal which we want to redefine in $M$. This implies that the values of the interactions in $\cg H$ are uniquely determined by the probabilities reconstructed in previous marginals and the log-linear interactions defined in $M$, thus they cannot be constrained.  Because $\cg H\subseteq \cg A$, certain constraints implied by the original conditional independence in $M$ cannot be implemented; in addition, when $\cg A\cap\cg R \neq \emptyset$, further limitations must be taken into account. In this section we study the nature and scope of the actual constraints that can be imposed in $M$.
We remind that a conditional independence statement, that holds only on a subset of the configurations of the conditioning variables, is a context specific conditional independence (\citet{Hojs:2004}).

The collection of interactions $\cg J=\cg H\cup (\cg A\cap\cg R)$ belongs to $\cg A$ but cannot be constrained in $M$; these interactions are of two kinds: (i) those which belong to $\cg H$ have to be omitted as a replacement for duplicating those in $\cg I$ and (ii)  those in $\cg R$ which we were unable to replicate because, if included in $\cg I$, there would not exist an $\cg H$ adequate for replacement. It follows that the conditional independence  in $M$ must be restricted to the context that does not require to constrain the collection  of  interactions $ \cg J$. More precisely, point (i) implies that the conditional independence can be defined only in the contexts in which, for every $(t,h) \in \cg H$ the variables in $h$ are different from $\b j_h$. Point (ii) implies that the conditional independence can be defined only in the contexts where, for every maximal set $m$ of $\cg I$ and $ v \in \cg A\cap\cg R$ the variables belonging to $v \setminus m$ are fixed to the reference category.
Let us examine some of the previous example to clarify the situation.
\begin{example}
Consider again example \ref{ex:6} here $\cg A\cap\cg R$ is empty and $$
\cg H = \{(12,5),(123,5),(13,5),(134,5)\};
$$
it follows that we are left with $1\ci (2,3)\mid (4,5)$, for all $X_5\ne j_5$. In example \ref{ex:5},  because $\cg J$ = $\{124,125,1234,1235,(12,45),(123,45)\}$, with $\cg{ I}$ = $\{12,123\}$,  we can have $1\ci 2\mid (3,4,5)$ where $X_4,X_5$ are fixed to the reference category, a statement of much more limited scope than the original one. Finally, in example \ref{ex:8}, though $\cg J$ = $\{125,126\}$ is smaller than before, with $\cg{I}$ = $\{12,124\}$, we can only impose $1\ci 2\mid (4,5,6)$ with $X_5,X_6$ fixed to the reference category.
\end{example}

As mentioned above, the conditionally independence in $M$ would not be restricted if the elements of $\cg H$ did not belong to $\cg A$, however, Theorem \ref{th:conv} implies that the resulting model is non smooth.
\begin{example}
In example \ref{ex:1} suppose that all variables have the same number of categories; because $\cg A$ = $\{12,123,124,1234\}$, the conditional independence is not affected if we take $\cg H$ = $\{23,234\}$; though this corresponds to the same number of parameters as $\cg I$, the jacobian $\b J$ of the LM algorithm is the identity matrix and the model is non smooth.
\end{example}
\subsection*{Acknowledgments}
We are grateful to M. Drton for useful discussions.
\section*{Appendix}
\subsection*{Interactions defined as contrasts of averages of logarithms of probabilities.}
An alternative to the {\it reference category} interactions $\eta_{I;M}(\b x_I)$
are the interactions based on contrasts of averages which may be defined as
\begin{equation}
\bar\eta_{I;M}(\b x_I \mid \b x_{M\setminus I})=\sum_{b\subseteq I}\frac{1}{\langle I \setminus b\rangle} (-1)^{|I\setminus b|}\sum_{\b x_{I \setminus b}}\log(\pi(\b x_b,\b x_{I\setminus b},\b x_{M\setminus I})),
\label{eq:anova}
\end{equation}
where $\langle I \rangle$ denotes the number of possible configurations of the vector $\b x_I$.

When $M \setminus I$ is not empty, the interactions defined in  (\ref{eq:anova})
depend on the value of the remaining variables and may be interpreted as the log-linear interaction $I$ in the marginal
distribution $M$ conditionally on $X_h=x_h\:\forall h\in \bs M I$.
To use the interactions defined in (\ref{eq:anova}) as a parameterization, the usual way of removing redundancies is to average with respect the conditioning variables, leading to the following expression
\begin{eqnarray*}
\bar\eta_{I;M}(\b x_I) &=& \frac{1}{\langle M \setminus I\rangle}\sum_{\bl x_{M\bksl I}}\bar\eta_{I;M}(\b x_I \mid \b x_{M\setminus I})\\
&=& \frac{1}{\langle M \setminus I\rangle}\sum_{\bl x_{M\bksl I}}\sum_{b\subseteq I}\frac{1}{\langle I \setminus b\rangle} (-1)^{| I\setminus b|}\sum_{\b x_{I \setminus b}}\log(\pi(\b x_b,\b x_{I\setminus b},\b x_{M\setminus I}))\\
&=& \frac{1}{\langle M \setminus b\rangle}\sum_{\bl x_{M\bksl b}}\sum_{b\subseteq I} (-1)^{| I\setminus b|}\log(\pi(\b x_b,\b x_{M\setminus b})).
\label{BRmed}
\end{eqnarray*}

It is well known that both  the {\it contrasts of averages} interactions $\bar\eta_{a;M}(\b x_a)$, used by \citet{BerRud:02}, and  the {\it reference category} interactions $\eta_{a;M}(\b x_a)=\eta_{a;M}(\b x_a\mid \b 0_{\bsl M a})$, used in this paper, are a parametrization of the joint probabilities.

Let
$\bar{\b\eta}(I,M;\b 0_{M\setminus I})$ denote the vector of log-linear interactions in (\ref{eq:anova}) when the variables in $\b x_I$ take all possible configurations in lexicographic order. It is easy to verify that we may write  $\bar{\b\eta}(I,M;\b 0_{M\setminus I})$ = $\b S(I,M)\log\b p(M)$, where $\b S(I,M)$ = $\bigotimes_{i\in M} \b S_i$ and $\b S_i$ is equal to the matrix $\b I-\b 1\b 1\tr/(r_i+1)$ without the first row if $i\in I$ and to the vector $(1\quad , \b 0\tr)$ otherwise. It is also easy to verify that the vector of log-linear interactions $\bar{\b\eta}(I,M)$, obtained by averaging over all possible configurations of the conditioning variables, may be written as $\bar {\b\eta}(I,M)$ = $\bar {\b  S}(I,M)\log \b p(M)$, where $\bar { \b S}(I,M)$ = $\bigotimes_{i\in M} \bar {\b S_i}$ and  $\bar{\b S_i}$ is equal to $\b S_i$ when $i\in I$ and to the vector $\b 1\tr/(r_i+1)$ otherwise.
\begin{lemma}\label{BRCF}
\begin{eqnarray}
\bar{\b\eta}(I,M;\b 0_{M\setminus I}) &=& \b A \b\eta(I,M)\label{eq:cond} \\
\bar{\b\eta}(I,M) &=& \b B \b\eta(\cg I,M) \label{eq:av}
\end{eqnarray}
for suitable matrices of constants $\b A$ and $\b B$.
\end{lemma}
\begin{proof}
By substitution in (\ref{reconstruct}),  $\bar{\b\eta}(I,M;\b 0_{M\setminus I})$ = $\b S(I,M)\b G(M) \b\eta(M)$ and (\ref{eq:cond}) follows by noting that the kronecker product contains a 0 factor if there is an $i\in I,\: i\not\in J$, because $\b S_i$ is a matrix of row contrasts and $\b G_i$ is the unitary vector; the same result arise if there is a $i\not\in I,\: i\in J$, because $\b S_i$ is the vector $(1\quad , 0\tr)$ and $\b G_i$ is the $\bar{\b I_i}$ matrix whose first row is a row of 0's. Equation(\ref{eq:av}) follows by a similar argument: when $i\in I,\:i\not\in J$, $\bar{ \b S_i}$ = $\b S_i$ and we get a 0 factor as above, instead, when $i\not \in I,\:i\in J$, $\bar{ \b S_i}$ is proportional to the unitary vector so that $\bar{ \b S_i} \b G_i$ is also proportional to a unitary vector.
\end{proof}
By noting that any category may be chosen as reference category for each variable, (\ref{eq:cond}) implies that any log-linear interaction $I$ defined in (\ref{eq:anova}) is a linear function of the log-linear interactions $I$ defined in (\ref{eq:mlli}) for all possible values of $\b x_I$. Instead, the log-linear interactions $\bar{\b\eta}(I,M)$, obtained by averaging across the conditioning variables, are a linear function of all $\b\eta(J,M)$, $I\subseteq J\subseteq M$.
\subsection*{Proofs of the Lemmas}
\begin{proof}
{\bf Proof of Lemma \ref{le:spectr}}.
Note that $\b A$ is the residual variance in a linear model where the binary variables indexed by $\cg H$ are regressed on the variables in $\cg R$ while $\b B$ is the residual variance when $\cg H$ is regressed on the variables in $\cg R$ and $\cg I$. Then, properties of linear projections imply that we may write $\b C$ = $\b A-\b B$ where
$$\b C = \b Q_{\cgl I\cgl H\mid\cgl R}\tr \left( \b F_{\cgl I\cgl
I}-\b F_{\cgl I\cgl R}\b F_{\cgl R\cgl R}^{-1}\b F_{\cgl R \cgl
I}\right)^{-1} \b Q_{\cgl I\cgl H\mid\cgl R}
$$
is clearly a positive semi-definite matrix. Then Theorem 7.7.3 in \citet{HornJ:2009} implies that the spectral radius of $\b B\b A^{-1}$, which is equal to that of $\b A^{-1}\b B$, is always not greater than 1. The spectral radius is exactly 1 if and only if $\b C$ or. equivalently,  $\b Q_{\cgl I\cgl H\mid\cgl R}$ is singular.
\end{proof}
\begin{proof}
{\bf Proof of Lemma \ref{le:noH}}. (i) { Suppose there is a $v:\:v\in \cg H,\:v\not\in \cg A$ and consider the intersections of $v$ with the elements of $\cg I\cup\cg R$; these must belong to $\cg V$ and cannot be contained in $\cg A$, hence} they must belong to $\cg R$; the argument in the proof of Theorem  \ref{th:ind} implies that the corresponding columns in $\b Q_{\cgl I \cgl H \mid \cgl R}$ are 0. (ii) { Under independence $\b P_a\b G_v =\bigotimes_{j=1}^{d} \b\Pi_j$, where the factors $\b\Pi_j$ are the entries of the last row of Table 1, $a \in (\cg I\cup \cg R)$ and $v\in \cg H$. If there is a variable $j\in v$ which is not contained in any element of $\cg I\cup \cg R$, the last entry in the forth column of Table 1 indicates that there will be a factor $\b 1_j\bar{\b\pi}_j\tr$ where $\bar{\b\pi}_j\tr$ has $r_j$ columns; the result follows from an argument similar to the one at the beginning of the proof of Theorem \ref{th:ind}.}
\end{proof}
\begin{table}[h]
\begin{center}
\caption{}
\begin{tabular}{lcccccc}
 & \multicolumn{3}{c}{$j\in a$} & \multicolumn{3}{c}{$j\not\in a$} \\
 & $j\in t$ & $j\in h$ & $j\not\in (t\cup h)$ & $j\in t$ & $j\in h$ & $j\not\in (t\cup h)$ \\
$\b P_a(j)$ & $\b I_j$ & $\b I_j$ & $\b I_j$ & $\b 1_j\b\pi_j\tr$ & $\b 1_j\b\pi_j\tr$ & $\b 1_j\b\pi_j\tr$ \\
$\b G_{t,h}(j)$ & $\bar{\b I}_j$ & $\b e_{jl}$ & $\b 1_j$ & $\bar{\b I}_j$ & $\b e_{jl}$ & $\b 1_j$ \\
$\Pi_j$ & $\bar{\b I}_j$ & $\b e_{jl}$ & $\b 1_j$ & $\b 1_j\bar{\b\pi}_j\tr$ & $\b 1_j\pi_{jh}$ & $\b 1_j$
\end{tabular}
\end{center}
\end{table}
\begin{proof}
{\bf Proof of Lemma \ref{le:ex}}.
When $\cg K(t)$ is not a singleton, by construction, the intersection of two or more elements of $\cg K(t)$ is disjoint from $h$, thus $\cg K(t,h)$ is formed by sets $\cg G$ with cardinality not smaller than two. Let $n_t$ is the cardinality of $\cg K(t)$ then:
$$
{
\sum_{\cgl G \in \cgl K(t,h)}(-1)^{|\cgl G|+1}=\sum_{i=2}^{n_t}\left( \begin{array}{c} n_t \\ i \\ \end{array} \right)  (-1)^{i+1} = -\sum_{i=0}^1 \left( \begin{array}{c} n_t \\ i \\ \end{array} \right)  (-1)^{i+1}=-n_t+1,
}
$$
thus, point (ii) of Definition 3 is satisfied. Point (iii) is trivially satisfied because $\cg I$ is a singleton. When $\cg K(t)$ is a singleton all the conditions of Definition 1 are trivially satisfied.
\end{proof}
\begin{proof}
{\bf Proof of Lemma \ref{le:prod}}.
Under independence $\b P_a \b G_{t,h}$ = $\bigotimes_{j=1}^{d} \b\Pi_j$, where $\b\Pi_j$ = $\b P_a(j) \b G_{t,h}(j)$; the possible values of $\b\Pi_j$ are given in Table 1 where $\b{\bar I}_j$ is an identity matrix without the first column, $\b 1_j$ is a vector of ones $\b e_{jl}$ a vector of 0's except for a 1 in the $l$th position, and $\b\pi_j$ is the marginal distribution of $X_j$, all of dimension $r_j+1$.
Point a) follows from the first two columns of Table 1, while b) follows from columns 1 and 5.
\end{proof}

\bibliographystyle{elsarticle-harv}
\bibliography{bib}

\end{document}